\documentclass[12pt,reqno]{amsart}
\usepackage[colorlinks=true, pdfstartview=FitV, linkcolor=blue, citecolor=blue, urlcolor=blue]{hyperref}
\usepackage{amssymb,amsmath, amscd}
\usepackage{ragged2e}
\usepackage{times, verbatim}
\usepackage{mathdots}
\usepackage{graphicx}
\usepackage[english]{babel}
 \usepackage[usenames, dvipsnames]{color}
\usepackage{amsmath,amssymb,amsfonts}
\usepackage{enumerate}
\usepackage{MnSymbol}
\usepackage{anysize}
\usepackage{enumitem}
\usepackage{bigints}
\usepackage{braket}
\usepackage[english]{babel}
\usepackage{blindtext}
\usepackage{mathtools}
\usepackage{graphicx}
\usepackage[a4paper,verbose]{geometry}
\marginsize{2.5cm}{2.5cm}{2.5cm}{2.5cm}
\input xy
\xyoption{all}
\usepackage{pb-diagram}
\usepackage[all]{xy}
\input xy
\xyoption{all}

\DeclareFontFamily{OT1}{rsfs}{}
\DeclareFontShape{OT1}{rsfs}{n}{it}{<-> rsfs10}{}
\DeclareMathAlphabet{\mathscr}{OT1}{rsfs}{n}{it}

\begin{document}
\theoremstyle{plain}

\newtheorem{theorem}{Theorem}[section]
\newtheorem{thm}[equation]{Theorem}
\newtheorem{prop}[equation]{Proposition}
\newtheorem{corollary}[equation]{Corollary}
\newtheorem{conj}[equation]{Conjecture}
\newtheorem{lemma}[equation]{Lemma}
\newtheorem{defn}[equation]{Definition}
\newtheorem{question}[equation]{Question}

\theoremstyle{definition}
\newtheorem{conjecture}[theorem]{Conjecture}
\newtheorem{coro}[theorem]{Corollary}

\newtheorem{example}[equation]{Example}
\numberwithin{equation}{section}

\newtheorem{remark}[equation]{Remark}

\newcommand{\bigboxplus}{
	\mathop{
		\vphantom{\bigoplus} 
		\mathchoice
		{\vcenter{\hbox{\resizebox{\widthof{$\displaystyle\bigoplus$}}{!}{$\boxplus$}}}}
		{\vcenter{\hbox{\resizebox{\widthof{$\bigoplus$}}{!}{$\boxplus$}}}}
		{\vcenter{\hbox{\resizebox{\widthof{$\scriptstyle\oplus$}}{!}{$\boxplus$}}}}
		{\vcenter{\hbox{\resizebox{\widthof{$\scriptscriptstyle\oplus$}}{!}{$\boxplus$}}}}
	}\displaylimits 
}

\newcommand{\Hecke}{\mathcal{H}}
\newcommand{\Liea}{\mathfrak{a}}
\newcommand{\Cmg}{C_{\mathrm{mg}}}
\newcommand{\Cinftyumg}{C^{\infty}_{\mathrm{umg}}}
\newcommand{\Cfd}{C_{\mathrm{fd}}}
\newcommand{\Cinftyfd}{C^{\infty}_{\mathrm{ufd}}}
\newcommand{\sspace}{\Gamma \backslash G}
\newcommand{\PP}{\mathcal{P}}
\newcommand{\bfP}{\mathbf{P}}
\newcommand{\bfQ}{\mathbf{Q}}
\newcommand{\Siegel}{\mathfrak{S}}
\newcommand{\g}{\mathfrak{g}}
\newcommand{\A}{\mathbb{A}}
\newcommand{\Q}{\mathbb{Q}}
\newcommand{\Gm}{\mathbb{G}_m}
\newcommand{\Nm}{\mathbb{N}m}
\newcommand{\ii}{\mathfrak{i}}
\newcommand{\II}{\mathfrak{I}}

\newcommand{\kk}{\mathfrak{k}}
\newcommand{\nn}{\mathfrak{n}}
\newcommand{\tF}{\widetilde{F}}
\newcommand{\p}{\mathfrak{p}}
\newcommand{\m}{\mathfrak{m}}
\newcommand{\bb}{\mathfrak{b}}
\newcommand{\Ad}{{\rm Ad}\,}
\newcommand{\ttt}{\mathfrak{t}}
\newcommand{\frakt}{\mathfrak{t}}
\newcommand{\U}{\mathcal{U}}
\newcommand{\Z}{\mathbb{Z}}
\newcommand{\bfG}{\mathbf{G}}
\newcommand{\bfT}{\mathbf{T}}
\newcommand{\R}{\mathbb{R}}
\newcommand{\ST}{\mathbb{S}}
\newcommand{\h}{\mathfrak{h}}
\newcommand{\bC}{\mathbb{C}}
\newcommand{\C}{\mathbb{C}}
\newcommand{\N}{\mathbb{N}}
\newcommand{\qH}{\mathbb {H}}
\newcommand{\temp}{{\rm temp}}
\newcommand{\Hom}{{\rm Hom}}
\newcommand{\Aut}{{\rm Aut}}
\newcommand{\rk}{{\rm rk}}
\newcommand{\Ext}{{\rm Ext}}
\newcommand{\End}{{\rm End}\,}
\newcommand{\Ind}{{\rm Ind}}
\newcommand{\ind}{{\rm ind}}
\newcommand{\Irr}{{\rm Irr}}
\def\circG{{\,^\circ G}}
\def\M{{\rm M}}
\def\diag{{\rm diag}}
\def\Ad{{\rm Ad}}
\def\As{{\rm As}}
\def\wG{{\widehat G}}
\def\G{{\rm G}}
\def\SL{{\rm SL}}
\def\PSL{{\rm PSL}}
\def\GSp{{\rm GSp}}
\def\PGSp{{\rm PGSp}}
\def\Sp{{\rm Sp}}
\def\St{{\rm St}}
\def\GU{{\rm GU}}
\def\SU{{\rm SU}}
\def\U{{\rm U}}
\def\GO{{\rm GO}}
\def\GL{{\rm GL}}
\def\PGL{{\rm PGL}}
\def\GSO{{\rm GSO}}
\def\GSpin{{\rm GSpin}}
\def\GSp{{\rm GSp}}

\def\Gal{{\rm Gal}}
\def\SO{{\rm SO}}
\def\O{{\rm  O}}
\def\Sym{{\rm Sym}}
\def\sym{{\rm sym}}
\def\St{{\rm St}}
\def\Sp{{\rm Sp}}
\def\tr{{\rm tr\,}}
\def\ad{{\rm ad\, }}
\def\Ad{{\rm Ad\, }}
\def\rank{{\rm rank\,}}

\def\Ext{{\rm Ext}}
\def\Hom{{\rm Hom}}
\def\Alg{{\rm Alg}}
\def\GL{{\rm GL}}
\def\SO{{\rm SO}}
\def\G{{\rm G}}
\def\U{{\rm U}}
\def\St{{\rm St}}
\def\Wh{{\rm Wh}}
\def\RS{{\rm RS}}
\def\ind{{\rm ind}}
\def\Ind{{\rm Ind}}
\def\csupp{{\rm csupp}}

\subjclass{Primary 11F70; Secondary 22E55}
\title{Non-vanishing of certain integral representations}
\author{Akash Yadav}
\address{Department of Mathematics \\
Indian Institute of Technology Bombay \\ Mumbai 400076}
\email{194090003@iitb.ac.in}
\begin{abstract}
In this paper, we prove that there exist Whittaker and Schwartz functions such that the local Flicker integrals are non-vanishing for all complex values of $s$, and the local Bump-Friedberg integrals are non-vanishing for all complex pairs $(s_1,s_2)$. As a corollary, we determine the potential locations of poles for their corresponding partial L-functions.
\end{abstract}
\maketitle
\section{Introduction}

Let \( F \) be a local field of characteristic zero and \( E \) be either \( F \times F \) (the split case) or a quadratic extension of \( F \) (the inert case). Fix a a non-trivial additive character \( \psi: E \to \mathbb{S}^1 \) trivial on \( F \). Let \( \pi \) be an irreducible generic representation of \( \GL_n(E) \) with  Whittaker model \( \mathcal{W}(\pi, \psi) \), where \( n \geq 1 \) is an integer. Denote by \( \mathcal{S}(F^n) \) the Schwartz space of \( F^n \), and let \( e_n = (0, 0, \ldots, 1) \in F^n \).  

Flicker \cite{Fli1988} (for the inert case) and Rankin-Selberg \cite{JPSS1983} (for the split case) introduced the local zeta integrals  
\[
Z(s, W, \Phi) = \int\limits_{N_n(F) \backslash \GL_n(F)} W(g) \Phi(e_n g) |\operatorname{det}(g)|_F^s \, dg,
\]
where \( s \) is a complex variable, \( N_n \) denotes the subgroup of unipotent upper triangular matrices in \( \GL_n \), and \( W \in \mathcal{W}(\pi, \psi) \), \( \Phi \in \mathcal{S}(F^n) \). These integrals are absolutely convergent for \( \operatorname{Re}(s) \gg 0 \) and admit meromorphic continuation to the entire complex plane (\cite{BP2021}).  

In \cite{HJ2024}, Humphries and Jo showed that archimedean newforms (introduced in \cite{HU2024}) serve as weak test vectors for several local $L$-functions. Using their computations, we show that it is possible to select a uniform pair of Whittaker and Schwartz functions that guarantee the non-vanishing of these local integrals for any complex number $s_0$. This constitutes the first main theorem of this paper.

\begin{thm}\label{main1}
    There exist $W\in\mathcal{W}(\pi,{\psi})$ and $\Phi\in\mathcal{S}(F^n)$ such that the function $s\mapsto Z(s,W,\Phi)$ does not vanish at any $s_0\in\mathbb{C}$.
\end{thm}

Note that this result is stronger than Lemma $3.3.3$ of Beuzart-Plessis \cite{BP2021} as a single choice of Whittaker and Schwartz function works for all complex numbers $s_0$.

Next, let $\pi$ be an irreducible generic representation of $\operatorname{GL}_n(F)$, where \(n = 2m\) or \(n = 2m+1\), depending on whether \(n\) is even or odd. Fix a non-trivial additive character $\psi': F \to \mathbb{S}^1$. For \(W \in \mathcal{W}(\pi, \psi')\) and \(\Phi \in \mathcal{S}(F^m)\), Bump and Friedberg \cite{BF1990} introduced the local Bump--Friedberg integrals \(B(s_1, s_2, W, \Phi)\)
\[
=\begin{cases}
     \int\limits_{N_m(F)\backslash G_m(F)}\int\limits_{N_m(F)\backslash G_m(F)}W(J(g,g'))\Phi(e_m g')\left|\operatorname{det}(g)\right|_F^{s_1-1/2}\left|\operatorname{det}(g')\right|_F^{s_2-s_1+1/2}\>dg\>dg'\\ \\
   \int\limits_{N_m(F)\backslash G_m(F)}\int\limits_{N_{m+1}(F)\backslash G_{m+1}(F)}W(J(g,g'))\Phi(e_{m+1} g)\left|\operatorname{det}(g)\right|_F^{s_1}\left|\operatorname{det}(g')\right|_F^{s_2-s_1}\>dg\>dg'
\end{cases}
\]
when $n$ is even and odd respectively. We refer the reader to Section $\ref{s2}$ for the definition of $J(g,g').$ These integrals converge absolutely for large enough $\operatorname{Re}(s_1)$ and $\operatorname{Re}(s_2)$ (\cite{BF1990}).

As before, we show that it is possible to select a uniform pair of Whittaker and Schwartz functions that guarantee the non-vanishing of these local integrals for any pair of complex numbers $(s_0,s_0')\in\mathbb{C}\times\mathbb{C}$. This forms the second main theorem of this paper.

\begin{thm}\label{main2}
    There exist a choice of $W\in\mathcal{W}(\pi,{\psi'})$ and $\Phi\in\mathcal{S}(F^m)$ such that the function $(s_1,s_2)\mapsto B(s_1,s_2,W,\Phi)$ admits a meromorphic continuation in the variables $s_1$ and $s_2$ and does not vanish at any $(s_0,s'_0)\in\mathbb{C}\times\mathbb{C}$.
\end{thm}

As a consequence of Theorems \ref{main1} and \ref{main2}, we can describe the meromorphic behavior of the corresponding partial $L$-functions. Let $L/K$ be a quadratic extension of number fields. For every place $\nu$ of $K$, we denote by $K_\nu$ the corresponding completion of $K$ and set $L_\nu=K_\nu\otimes_K L$. Let $\mathbb{A}_K=\prod'_{\nu} K_\nu$ and $\mathbb{A}_{L}=\mathbb{A}_K\otimes_K L=\prod_\nu' L_\nu$ be the ad\`ele rings of $K$ and $L$ respectively.

Let \( \Pi \) be a unitary cuspidal automorphic representation of $\operatorname{GL}_n(\mathbb{A}_L)$ with restricted tensor product decomposition \( \bigotimes'_\nu \Pi_\nu \) over the places \( \nu \) of \( K \) and central character $\omega_\Pi$. If $\omega_{\Pi}\vert_{\mathbb{A}^{\times}_K}$ is trivial on ideles of $K$ of norm $1$, let $\delta$ be the real number such that $\omega_{\Pi}\vert_{\mathbb{A}^{\times}_K}(.) = \left|.\right|^{in\delta}$. Let $S$ be any finite set of places of $K$ containing the archimedean places. We define the partial Asai $L$-function of $\Pi$ as
$$L^S(s,\Pi,\mathrm{As})\coloneqq\prod\limits_{\nu\notin S}L(s,{\Pi}_{\nu},\mathrm{As}).$$
Here, $L(s, {\Pi}_{\nu}, \mathrm{As})$ denotes the Rankin-Selberg $L$-function of ${\Pi}_{\nu}$ when $\nu$ splits over $E$, and the Asai $L$-function of ${\Pi}_{\nu}$ when $\nu$ is inert in $E$, both defined via the local Langlands correspondence. Using Theorem $\ref{main1}$, we obtain the following corollary.
\setcounter{theorem}{2}
\begin{coro}\label{coro1}
       The meromorphic function \( L^S(s,\Pi,\mathrm{As}) \) is entire if \( \omega_{\Pi} \vert_{\mathbb{A}^{\times}_K} \) is nontrivial on ideles of $K$ of norm \( 1 \). Otherwise, it can have at most simple poles at \( s = -i\delta \) and \( s = 1 - i\delta \).
\end{coro}
We note that Corollary $\ref{coro1}$ appears in the literature only under the assumption that the archimedean places of $K$ split over $L$ and the set $S$ is the complement of all the places of $K$ where $\Pi_\nu$ is unramified (see \cite{Fli1988}, \cite{Fli1993}). However, it may be familiar to experts in the field.

\vspace{0.2 cm}

Next, let $\Pi$ be a unitary cuspidal automorphic representation of $\operatorname{GL}_n(\mathbb{A}_K)$ with restricted tensor product decomposition \( \bigotimes'_\nu \Pi_\nu \) over the places \( \nu \) of \( K \) and central character $\omega_\Pi$. If $\omega_{\Pi}$ is trivial on ideles of $K$ of norm $1$, let $\delta$ be the real number such that $\omega_{\Pi}(.) = \left|.\right|^{in\delta}$. Let $T$ be a finite set of places of $K$ that includes the archimedean places. We define the partial Bump–Friedberg \( L \)-function of \( \Pi \) as  
\begin{align*}
    L^T(s,\Pi,\mathrm{BF}) 
    &= \prod\limits_{\nu\notin T}L(s,{\Pi}_{\nu},\mathrm{BF}).
\end{align*}
Here, the local Bump–Friedberg \( L \)-function \( L(s,{\Pi}_{\nu},\mathrm{BF}) \) is the product \( L(s,{\Pi}_{\nu})L(2s,{\Pi}_{\nu},\wedge^2) \), where \( L(s,{\Pi}_{\nu}) \) and \( L(s,{\Pi}_{\nu},\wedge^2) \) are the standard and exterior square \( L \)-functions of \( \Pi_{\nu} \), respectively, both defined via the local Langlands correspondence. Using Theorem $\ref{main2}$, we derive the following corollary. 

\begin{coro}\label{coro2}{$\phantom{}$}
\vspace{0.2 cm}
\begin{enumerate}
\item For \( n \) even, the meromorphic function \( L^T(s,\Pi,\mathrm{BF}) \) is entire if \( \omega_{\Pi} \) is nontrivial on ideles of norm \( 1 \). Otherwise, it can have at most simple poles at \( s = -\frac{i\delta}{2} \) and \( s = \frac{1 - i\delta}{2} \).
 \item For $n$ odd, $L^T(s,\Pi,\mathrm{BF})$ is entire.
\end{enumerate}
\end{coro}

A similar result for the partial exterior square \( L \)-functions appears in the unpublished preprint of Belt \cite{BELT2011}, where he includes all ramified places in \( T \). His proof relies on the non-vanishing of the Jacquet–Shalika integrals introduced in \cite{JS1990}. 

We briefly outline the content of each section in this paper. In Section \ref{s2}, we establish the basic notation and preliminaries essential for the subsequent discussions. Section \ref{s3} introduces the theory of global and local Flicker integrals, culminating in the proof of Theorem $\ref{main1}$ and Corollary $\ref{coro1}$. Finally, in Section \ref{s4}, we review the theory of global and local Bump-Friedberg integrals, leading to the proofs of Theorem $\ref{main2}$ and Corollary $\ref{coro2}$.

\section{Preliminaries}\label{s2}

\subsection{Basic Notation}

Let $F$ be a local field of characteristic zero and $E$ be either a quadratic extension of $F$ (the inert case) or $F\times F$ (the split case). We write $|\cdot|_F$ and $|\cdot|_E$ for the normalized absolute values of $F$ and $E$ respectively and let $|\cdot|$ denote $|\cdot|^{1/2}_\mathbb{C}$. Thus, in the split case we have $|(\lambda, \mu)|_E=$ $|\lambda|_F|\mu|_F$ for every $(\lambda, \mu) \in E$ and in both cases we have $|x|_E=|x|_F^2$ for every $x \in F$. In the non-archimedean case, let $\mathcal{O}_F$ and $\mathcal{O}_E$ be the rings of integers of $F$ and $E$, respectively, with $q_F$ and $q_E$ denoting the cardinalities of their residue fields. 
Let $n\geq 1$ be an integer. Let $G_n$ denote the reductive group $\GL_n$ with $Z_n$, $B_n$, and $N_n$ being the subgroups of scalar, upper triangular, and unipotent upper triangular matrices in $G_n$ respectively.  Let \( P_n \) be the mirabolic subgroup of \( G_n \), consisting of matrices that stabilize the row vector \( e_n = (0, \dots, 0, 1) \) under the right multiplication action. Let $\mathbb{S}^1$ denote the multiplicative group of all complex numbers with absolute value $1$. Let $^{\iota}g$ denote the inverse transpose of an element $g$ in $G_n$ and $w_n$ denote the $n\times n$ matrix $\left( \begin{array}{ccc} &&1 \\ & \iddots & \\ 1&& \end{array} \right)$. 

 For $n=2m$, we define the embedding $J: G_m \times G_m \to G_n$ by
\[
J(g,g^{\prime})_{k,\ell} \coloneqq \begin{dcases*}
g_{i,j} &  if $k=2i-1$ and $\ell=2j-1$, \\
g^{\prime}_{i,j} &  if $k=2i$ and $l=2j$, \\
0 &  otherwise.
\end{dcases*}
\] 
Similarly, for $n = 2m + 1$, we define the embedding $J: G_{m+1} \times G_m \to G_n$ by
\[
J(g,g^{\prime})_{k,\ell} \coloneqq \begin{dcases*}
g_{i,j} & if $k=2i-1$ and $\ell=2j-1$, \\
g^{\prime}_{i,j} & if $k=2i$ and $\ell=2j$,\>\text{\>and} \\
0 &  otherwise.
\end{dcases*}
\]
Fix non-trivial additive characters $\psi^{\prime}: F \rightarrow \mathbb{S}^1$ and $\psi: E \rightarrow \mathbb{S}^1$, and assume that $\psi$ is trivial on $F$. We define generic characters $\psi_n: N_n(E) \rightarrow \mathbb{S}^1, \psi_n^{\prime}: N_n(F) \rightarrow \mathbb{S}^1$ by
\[
\psi_n^{\prime}(u)=\psi^{\prime}\left((-1)^n \sum_{i=1}^{n-1} u_{i, i+1}\right) \text { and } \psi_n(u)=\psi\left((-1)^n \sum_{i=1}^{n-1} u_{i, i+1}\right) .
\]
In the split case, we set $\tau$ to be the unique element $(\beta,-\beta)\in F^{\times}\times F^{\times}$ such that $$\psi(x,y)=\psi^{\prime}(\beta x)\psi^{\prime}(-\beta y).$$ In the inert case, we set $\tau$ to be the unique element in $E$ such that $\psi(z)=\psi^{\prime}(\operatorname{Tr}_{E/F}(\tau z))$ for every $z\in E$, where $\operatorname{Tr}_{E/F}$ stands for the trace of the extension $E/F$. 
Let $\mathcal{S}\left(F^n\right)$ be the space of Schwartz functions on $F^n$. We denote by $\Phi \mapsto \widehat{\Phi}$ the Fourier transform on $F^n$ defined as follows : for every $\Phi \in \mathcal{S}\left(F^n\right)$ we have
$$
\widehat{\Phi}\left(x_1, \ldots, x_n\right)=\int_{F^n} \Phi\left(y_1, \ldots, y_n\right) \psi^{\prime}\left(x_1 y_1+\ldots+x_n y_n\right) d y_1 \ldots d y_n  
$$
for all $\left(x_1, \ldots, x_n\right) \in F^n$, where the measure of integration is chosen so that $\widehat{\widehat{\Phi}}(v)=\Phi(-v)$.

By a {\em representation} of $G_n(F)$, we will always mean a smooth representation of finite length with complex coefficients. Here {\em smooth} has the usual meaning in the non-archimedean case (i.e. every vector has an open stabilizer) whereas in the archimedean case it means a smooth admissible Fr\'echet representation of moderate growth in the sense of Casselman-Wallach. We let $\Irr(G_n(F))$ and $\Pi_2(G_n(F))$ be the sets of isomorphism classes of all irreducible representations and irreducible square-integrable representations of $G_n(F)$ respectively.

\subsection{Essentially square-integrable representations and parabolic induction}

Let \(P\) be a standard parabolic subgroup of \(G_n\) and \(MU\) be its Levi decomposition. Then, \(M\) can be expressed as $M = G_{n_1} \times \ldots \times G_{n_k}$ for some integers \(n_1, \ldots, n_k\) such that \(n_1 + \ldots + n_k = n\). For each \(i\) with \(1 \leq i \leq k\), let \(\tau_i \in \operatorname{Irr}(G_{n_i}(F))\), so that the representation $\sigma = \tau_1 \otimes \ldots \otimes \tau_k$
is an irreducible representation of \(M(F)\). We denote the normalized induced representation by 
\[
i_{P(F)}^{G_n(F)}(\sigma)=\tau_1 \times \ldots \times \tau_k.
\]
A representation $\pi\in \Irr(G_n(F))$ is {\em generic} if it admits a nonzero Whittaker functional with respect to any (or equivalently one) generic character of $N_n(F)$. We will denote by $\Irr_{gen}(G_n(F))$ the subset of generic representations in $\Irr(G_n(F))$. By \cite[Theorem 9.7]{Zel1980} and \cite[Theorem 6.2f]{Vog1978}, every $\pi\in \Irr_{gen}(G_n(F))$ is isomorphic to a representation of the form $\tau_1\times\ldots \tau_k$ where for each $1\leqslant i\leqslant k$, $\tau_i$ is an essentially square-integrable (i.e. an unramified twist of a square-integrable) representation of some $G_{n_i}(F)$. If $\pi$ is an irreducible generic representation of $G_n(F)$, we write $\mathcal{W}(\pi,\psi_n)$ for its Whittaker model (with respect to $\psi_n$).


We recall the local theory of essentially square-integrable representations in the archimedean case. Essentially square-integrable representations of $G_n(\C)$ exist only for $n = 1$. An essentially square-integrable representation of $G_1(\C)$ must be a character of the form $\pi(x)=e^{i\kappa \operatorname{arg}(x)} |x|^t_{\C}$ for some $\kappa \in \mathbb{Z}$ and $t \in \C$, where $e^{i\operatorname{arg}(x)} \coloneqq x/|x|$. Essentially square-integrable representations of $G_n(\R)$ exist only for $n\in\{ 1, 2\}$. An essentially square-integrable representation of $G_1(\mathbb{R})=\mathbb{R}^{\times}$ is a character of the form $\pi(x)=\operatorname{sgn}(x)^{\kappa}|x|^t_{\mathbb{R}}$ for some $\kappa \in \{0,1 \}$ and $t \in \C$, where $\operatorname{sgn}(x) \coloneqq x/|x|_{\mathbb{R}}$. 
We identify \( G_1(\mathbb{C}) \) as a subgroup of \( G_2(\mathbb{R}) \) via the mapping \( a + ib \mapsto \begin{psmallmatrix} a & b \\ -b & a \end{psmallmatrix} \). For \( \kappa \neq 0 \), the essential discrete series representation of weight \( |\kappa|_{\mathbb{R}} + 1 \), defined by the induced representation
\[
D_{|\kappa| + 1} \otimes \left|\det\right|^t_{\mathbb{R}} \coloneqq i_{G_1(\mathbb{C})}^{G_2(\mathbb{R})} e^{i\kappa \operatorname{arg}} |\cdot|^t_{\mathbb{C}} \cong i_{G_1(\mathbb{C})}^{G_2(\mathbb{R})} e^{-i\kappa \operatorname{arg}} |\cdot|^t_{\mathbb{C}},
\]
is essentially square-integrable.
Every essentially square-integrable representation of $G_2(\R)$ is of the form $\pi=D_{\kappa}\otimes \left|\det\right|^t_{\R}$ for some integer $\kappa \geq 2$ and $t \in \C$.

\subsection{Some global notation}

Let $L/K$ be a quadratic extension of number fields. For every place $\nu$ of $K$, we denote by $K_\nu$ the corresponding completion of $K$ and set $L_\nu=K_\nu\otimes_K L$. If $\nu$ is non-archimedean, we let $\mathcal{O}_{K_\nu}$ and $\mathcal{O}_{L_\nu}$ be the rings of integers of $K_\nu$ and $L_\nu$ respectively. Let $\mathbb{A}_K=\prod'_{\nu} K_\nu$ and $\mathbb{A}_{L}=\mathbb{A}_K\otimes_K L=\prod_\nu' L_\nu$ be the ad\`ele rings of $K$ and $L$ respectively and $| \cdot|_{\mathbb{A}_K}$ be the normalized absolute value on $\mathbb{A}_K$. Define $\mathbb{I}^1_K$ as the subgroup of the idele group $\mathbb{A}^\times_K$ given by $\mathbb{I}^1_K = \{ x \in \mathbb{A}^\times_K : |x|_{\mathbb{A}_K} = 1 \}$. For the sake of notational brevity, we write $[G] \coloneqq Z(\A_K) G(K) \backslash G(\A_K)$ for any reductive group $G$ with center $Z$.
Let \(\Psi'\) and \(\Psi\) be nontrivial additive characters of \(K \backslash \mathbb{A}_K\) and \(L \backslash \mathbb{A}_L\), respectively, with \(\Psi\) being trivial on \(K \backslash \mathbb{A}_K\). For every place \(\nu\) of \(K\), let \(\Psi'_{\nu}\) and \(\Psi_{\nu}\) be the local components of \(\Psi'\) and \(\Psi\) at \(\nu\), respectively. To each, we associate a generic character \(\Psi_{n,\nu}: N_n(L_\nu) \to \mathbb{S}^1\) and \(\Psi'_{n,\nu}: N_n(K_\nu) \to \mathbb{S}^1\), as before. Then, \(\Psi_n = \prod_\nu \Psi_{n,\nu}\) defines a character of \(N_n(\mathbb{A}_L)\), which is trivial on both \(N_n(L)\) and \(N_n(\mathbb{A}_K)\), and we define \(\Psi'_n = \prod_\nu \Psi'_{n,\nu}\).
 Let \( \mathcal{S}(\mathbb{A}^n_K) \) denote the Schwartz-Bruhat space on \( \mathbb{A}^n_K \). The Fourier transform on \( \mathbb{A}^n_K \), denoted by \( \Phi \mapsto \widehat{\Phi} \), is defined as follows: for every \( \Phi \in \mathcal{S}(\mathbb{A}^n_K) \), we have
\[
\widehat{\Phi}(x_1, \ldots, x_n) = \int_{\mathbb{A}^n_K} \Phi(y_1, \ldots, y_n) \Psi'(x_1 y_1 + \cdots + x_n y_n) \, d y_1 \cdots d y_n,
\]
for all \( (x_1, \ldots, x_n) \in \mathbb{A}^n_K \), where the measure of integration is chosen so that $\widehat{\widehat{\Phi}}(v)=\Phi(-v)$. 
Given a Schwartz-Bruhat function $\Phi \in \mathcal{S}(\mathbb{A}^n_K)$, we form the $\Theta$-series
\[
\Theta_{\Phi}(a,g) \coloneqq \sum_{\xi \in K^n} \Phi(a \xi g) \quad \text{for $a \in \A_K^{\times}$ and $g \in G_n(\A_K)$}.
 \]
Associated with this $\Theta$-series is an Eisenstein series, which is essentially the Mellin transform of $\Theta$. To be precise, for a unitary Hecke character 
$\eta : K^{\times} \backslash \A_K^{\times} \rightarrow \C^{\times}$, we set 
\[E(g,s;\Phi,\eta) \coloneqq \left|\det g\right|^s_{\A_K} \int_{K^{\times} \backslash \A_K^{\times}} \Theta'_{\Phi}(a,g) \eta(a)|a|^{ns}_{\A_K} \, d^{\times}a,
\]
where $\Theta'_{\Phi}(a,g) \coloneqq \Theta_{\Phi}(a,g) - \Phi(0)$. 

In \cite{JS1981}, Jacquet and Shalika established the analytic properties of the Eisenstein series.
\begin{thm}
  The Eisenstein series \(E(g,s;\Phi,\eta)\) has a meromorphic continuation to all of \(\mathbb{C}\). It is entire unless \(\eta\) is trivial on \(\mathbb{I}^1_K\) of the form \(\eta(a) = |a|^{in\delta}\) with \(\delta \in \mathbb{R}\), in which case it has at most simple poles at \(s = -i\delta\) and \(s = 1 - i\delta\). As a function of $g$ it is smooth of moderate growth and as a function of $s$ it is bounded in vertical strips (away from possible poles), uniformly for $g$ in compact sets. Moreover, it satisfies the functional equation
  \[
  E(g,s;\Phi,\eta) = E(^{\iota}g,1-s;\hat{\Phi},\eta^{-1}).
  \]
\end{thm}

\section{Flicker Integrals}\label{s3}

Let \((\Pi, V_{\Pi})\) be a unitary cuspidal automorphic representation of \( G_n(\mathbb{A}_L) \) with central character \( \omega_{\Pi} \). Then \( \Pi \) is isomorphic to a restricted tensor product, \( \Pi \cong \bigotimes'_\nu \Pi_\nu \), taken over the places \( \nu \) of \( K \). Here, each \( \Pi_\nu \) belongs to \( \operatorname{Irr}_{gen}(G_n(L_\nu)) \) and is unramified for all but finitely many places \( \nu \). For $\Phi\in\mathcal{S}(\mathbb{A}^n_K)$ and $\varphi\in V_{\Pi}$, Flicker \cite{Fli1988} defined the global integral
$$ I(s,\Phi,\varphi)=\int\limits_{[G_n]}E(g,s;\Phi,\omega_{\Pi}\vert_{\mathbb{A}^{\times}_K})\>\varphi(g)dg,$$

If $\omega_{\Pi}\vert_{\mathbb{A}^{\times}_K}$ is trivial on $\mathbb{I}^{1}_K$, let $\delta$ be the real number such that $\omega_{\Pi}\vert_{\mathbb{A}^{\times}_K}(.) = \left|.\right|^{in\delta}$.
  For ease of reference, we collect certain properties of these integrals (\cite{Fli1988}, \cite{KA2004}).
\begin{prop}\label{1.1}
    The integral $I(s,\Phi,\varphi)$ is convergent whenever the Eisenstein series is holomorphic at $s$. It has a meromorphic continuation to the entire complex plane and satisfies the functional equation
    \[I(s,\Phi,\varphi)=I(1-s,\hat{\Phi},\tilde{\varphi}),\]
    where $\tilde{\varphi}(g)=\varphi(w_n {^{\iota}g}).$
\end{prop}
In \cite{Fli1988}, Flicker showed that the poles of $ I(s,\Phi,\varphi)$ are closely related to those of the Eisenstein series.
\begin{prop}\label{1.2}
    The integral \( I(s,\Phi,\varphi) \) is entire if \( \omega_{\Pi} \) is nontrivial on \( \mathbb{I}^{1}_K \). Otherwise, it has at most simple poles at \( s = -i\delta \) and \( s = 1 - i\delta \).
\end{prop}

We recall the factorization of the integral $I(s,\Phi,\varphi)$ (as in Sections $2$ and $3$ of \cite{Fli1988}).

\begin{prop}\label{1.3}
    If $\varphi\in V_{\Pi}$ is a cusp form, let
    $$ W_{\varphi}(g)=\int\limits_{N_n(L)\backslash N_n(\mathbb{A}_{L})}\varphi(ng)\overline{\Psi}(n)\>dn $$
    be the associated Whittaker function. For $\Phi\in\mathcal{S}(\mathbb{A}^n_{K})$, the integral 
    $$ Z(s,W_\varphi,\Phi)=\int\limits_{N_n(\mathbb{A}_{K})\backslash G_n(\mathbb{A}_{K})}W_{\varphi}(g)\Phi(e_n g)\left|\operatorname{det}(g)\right|_{\mathbb{A}_K}^s\>dg$$
    converges absolutely and uniformly on compact sets when $\operatorname{Re}(s)$ is sufficiently large. When this is the case, we have
    \[I(s,\Phi,\varphi)=Z(s,W_\varphi,\Phi).\]
\end{prop}
The global integrals decompose into products of local integrals for decomposable vectors. Let  
$W_{\varphi} = \prod_{\nu} W_\nu$,  
where the product runs over all places $\nu$ of $K$, and each $W_\nu$ belongs to $\mathcal{W}(\Pi_\nu, \Psi_{n,\nu})$. For almost all unramified places $\nu$, assume that $W_\nu$ is the normalized spherical Whittaker function, defined as the unique Whittaker function that is invariant under $G_n(\mathcal{O}_{L_\nu})$ and satisfies $W_\nu(1) = 1$.

 Similarly, let \( \Phi = \prod_{\nu} \Phi_\nu \), where each \( \Phi_\nu \) is a Schwartz function in \( \mathcal{S}(K_\nu^n) \), and for almost all unramified places \( \nu \), \( \Phi_\nu \) is the characteristic function of \( \mathcal{O}_{K_{\nu}}^n \). When $\operatorname{Re}(s)$ is sufficiently large,
\[Z(s,W_\varphi,\Phi)=\prod\limits_{\nu}Z(s,W_\nu,\Phi_\nu),\]
where
\[Z(s,W_\nu,\Phi_\nu)=\int\limits_{N_n(K_{\nu})\backslash G_n(K_\nu)}W_{\nu}(g)\Phi_{\nu}(e_n g)|\operatorname{det}(g)|^s_{K_\nu}\>dg.\]
When a place $\nu$ of $K$ splits in $L$, this integral coincides with the Rankin-Selberg integral (see \cite{JPSS1983}). For $\Pi_\nu=\pi_1\otimes\pi_2\in\operatorname{Irr}_{\operatorname{gen}}\left(G_n(K_\nu\times K_\nu)\right)$ and $\tau=(1,-1)$, every $W_\nu\in\mathcal{W}\left(\Pi_\nu, \Psi_{n,\nu}\right)$ can be written as $W_1\otimes W_2$ for some $W_1\in\mathcal{W}\left(\pi_1, \Psi^{\prime}_{n,\nu}\right)$ and $W_2\in\mathcal{W}\left(\pi_2, \Psi^{\prime\,-1}_{n,\nu}
\right).$ Then for $\Phi_\nu\in\mathcal{S}(K^n_\nu)$, the integral $Z(s,W_\nu,\Phi_\nu)$ equals
\[Z(s,W_\nu,\Phi_\nu)=\int_{N_n(K_\nu) \backslash G_n(K_\nu)} W_1(g)W_2(g)\Phi_\nu(e_n g)|\operatorname{det}(g)|_{K_\nu}^s \>d g.\]

The integrals $Z(s, W_\nu,\Phi_\nu)$ converge absolutely when $\operatorname{Re}(s)$ is sufficiently large and admit a meromorphic continuation to the entire complex plane (see \cite{JPSS1983} for the split case and \cite{Fli1988}, \cite{BP2021} for the inert case).

Next, we present the unramified computation of these local integrals as discussed in the literature. Specifically, \cite{JPSS1983} addresses the split case, while \cite{Fli1988} covers the inert case. For a more recent exposition, see \cite{BP2021}.
\begin{prop}\label{3.5}
    Let \(\nu\) be a place of \(K\), and suppose that \(\Pi_{\nu}\) is an unramified, non-archimedean representation. Consider the normalized spherical Whittaker function \( W_\nu \in \mathcal{W}(\Pi_\nu, \Psi_{n,\nu}) \) and let \( \Phi_{\nu} \in \mathcal{S}(K^n_\nu) \) be the characteristic function of \( \mathcal{O}_{K_{\nu}}^n \). Then, we have:  

    \[
    Z(s, W_\nu, \Phi_\nu) =  L(s, \Pi_{\nu}, \mathrm{As}).
    \]
\end{prop} 

We are now ready to prove Theorem \ref{main1}.


\begin{flushleft}
    {\section*{Proof of Theorem \ref{main1}}}
\end{flushleft}

First, assume that \(E\) and \(F\) are non-archimedean. By Theorem 4.26 of \cite{MAT2009} and Proposition 9.4 of \cite{JPSS1983}, there exist \(W \in \mathcal{W}(\pi, \psi_{n})\) and \(\Phi \in \mathcal{S}(F^n)\) such that  
\[ 
Z(s,W,\Phi) = L_f(s, \pi, \mathrm{As}).
\]  
Here, 
$L_f(s,\pi,\mathrm{As})$ denotes the formal Asai $L$-function of $\pi$ when $E$ is inert over $F$ and the formal Rankin-Selberg $L$-function of $\pi$ when $E$ is split over $F$ (as defined in \cite{Jo2023}). As both of these formal $L$-functions are inverse of certain polynomials in $q_F^{-s}$ and $q_E^{-s}$, they are non-vanishing for all $s$.

Next, we assume that $E$ and $F$ are archimedean. First, we deal with the case when $E$ is split over $F$. If $\pi=\rho\otimes\sigma$ for irreducible generic representations $\rho=\rho_1\times\rho_2\cdots\times\rho_r$ and $\sigma=\sigma_1\times\sigma_2\cdots\times\sigma_l$ of $G_n(F)$ with $\rho_i$ and $\sigma_j$ essentially square-integrable, using Theorem $5.3$ and Proposition $5.5$ of \cite{HJ2024}, there exists a non-zero complex polynomial $p_{\rho\times\sigma,\mathrm{RS}}$ (dependent on $\pi$), $W\in\mathcal{W}(\pi,{\psi}_{n})$ and $\Phi\in\mathcal{S}(F^n)$ such that
\[{Z(s,W,\Phi)}=p_{\rho\otimes\sigma,\mathrm{RS}}(s) L(s,\rho\times\sigma).\]
We claim that the product $L(s,\rho\times\sigma)p_{\rho\otimes\sigma,\mathrm{RS}}(s)$ is non-vanishing for all $s$. Assume that $\rho\otimes\sigma$ is ramified otherwise $p_{\rho\otimes\sigma,\mathrm{RS}}\equiv 1$, and then our claim would hold.
By the multiplicativity of $L$-factors \cite{KN1994},
\[L(s,\rho\times\sigma) = \prod_{j=1}^{r} \prod_{i=1}^{l} L(s,\rho_j\times\sigma_i).\]
Therefore, by the proof of Proposition $5.5$ of \cite{HJ2024}, 
\[p_{\rho\otimes\sigma,\mathrm{RS}}(s)=\prod_{j=1}^{r} \prod_{i=1}^{l} p_{\rho_j\times\sigma_i,\mathrm{RS}}(s).\]
where $p_{\rho_j\times\sigma_i, \mathrm{RS}}$ is the polynomial as defined in Proposition $5.5$ of \cite{HJ2024}.
So we can restrict ourselves to the case where both $\rho$ and $\sigma$ are essentially square-integrable.

Suppose first that $F=\mathbb{C}$, so that if $\rho = e^{i\kappa \arg} |\cdot|_{\C}^{t}$, and $\sigma = e^{i\lambda \arg} |\cdot|_{\C}^{u}$, then \[L(s,\rho\times\sigma)p_{\rho\otimes\sigma,\mathrm{RS}}(s)=\zeta_{\mathbb{C}}\left(s+t+u+\frac{|\kappa|}{2}+\frac{|\lambda|}{2}\right).\]
Next, assume that $F = \R$. If $\rho = \operatorname{sgn}^{\kappa} |\cdot|_{\R}^{t}$ and $\sigma = \operatorname{sgn}^{\lambda} |\cdot|_{\R}^{u}$, then $$L(s,\rho\times\sigma)p_{\rho\otimes\sigma,\mathrm{RS}}(s)=\zeta_{\R}(s+t+\kappa+\lambda).$$ 
If $\rho= D_{\kappa} \otimes \left|\det\right|_{\R}^{t}$ and $\sigma = \operatorname{sgn}^{\lambda} |\cdot|_{\R}^{u}$, then
\[L(s,\rho\times\sigma)p_{\rho\otimes\sigma,\mathrm{RS}}(s)=\zeta_{\R}\left(s + t + u + \frac{\kappa - 1}{2} + \lambda\right)\zeta_{\R}\left(s + t + u + \frac{\kappa + 1}{2} + \lambda\right).\]
Finally, if $\rho = D_{\kappa} \otimes \left|\det\right|_{\R}^{t}$, and $\sigma = D_{\lambda} \otimes \left|\det\right|_{\R}^{u}$, then $L(s,\rho\times\sigma)p_{\rho\otimes\sigma,\mathrm{RS}}(s)$ equals
\[\zeta_{\R}\left(s + t + u + \frac{\kappa + \lambda}{2}+ 1\right)\zeta_{\R}\left(s + t + u + \frac{\kappa + \lambda}{2} - 1\right)\left[\zeta_{\R}\left(s + t + u + \frac{\kappa + \lambda}{2}\right)\right]^2.\]
In each of these cases, $L(s,\rho\times\sigma)p_{\rho\otimes\sigma,\mathrm{RS}}(s)$ does not vanish at any $s_0\in\mathbb{C}$.

Let us turn to the case when $E$ is inert over $F$. Then for $\pi=\pi_1\times\pi_2\cdots\times\pi_r$ with $\pi_i$ essentially square-integrable, using Theorem $6.6$ and Proposition $6.7$ of \cite{HJ2024}, there exists a non-zero complex polynomial $p_{\pi,\mathrm{As}}$ (dependent on $\pi$), $W\in\mathcal{W}(\pi,{\psi}_{n})$ and $\Phi\in\mathcal{S}(F^n)$ such that
\[{Z(s,W,\Phi)}=p_{\pi,\mathrm{As}}(s) L(s,\pi,\mathrm{As}).\]
We claim that the product $L(s,\pi,\mathrm{As})p_{\pi,\mathrm{As}}(s)$ is non-vanishing for all $s$. Assume that $\pi$ is ramified otherwise $p_{\pi,\mathrm{As}}\equiv 1$, and then our claim would hold.
We have that
\[L(s,\pi,\mathrm{As}) = \prod_{j = 1}^{r} L(s,\pi_j,\mathrm{As}) \prod_{1 \leq j < \ell \leq r} L(s, \pi_j \times \pi_{\ell})\]
via the multiplicativity of $L$-factors (Lemma $3.2.1$, \cite{BP2021}). Therefore, by the proof of Proposition $6.7$ of \cite{HJ2024}, 
\[p_{\pi,\mathrm{As}}(s)= \prod_{j = 1}^{r}p_{\pi_j,\mathrm{As}}(s) \prod_{1 \leq j < \ell \leq r} p_{\pi_j\times\pi_l,\mathrm{RS}}(s).\]
where $p_{\pi_j, \mathrm{As}}$ is the polynomial as defined in Proposition $6.7$ of \cite{HJ2024}. 
So we can restrict ourselves to the case where $\pi$ is essentially square-integrable.
If $\pi = e^{i\kappa \arg} |\cdot|_{\C}^{t}$, then
\[p_{\pi,\mathrm{As}}(s)L(s,\pi,\mathrm{As})=\zeta_{\R}(s+2t+|\kappa|),\]
which is non-vanishing at any $s_0\in\mathbb{C}$. This completes the proof.
\vspace{0.5 cm}
\section*{Proof of Corollary \ref{coro1}}

 We reiterate that, for notational convenience, we denote by \( L(s, \Pi_{\nu}, \mathrm{As}) \) the Rankin-Selberg \( L \)-function of \( \Pi_\nu \) in the case when \( \nu \) splits over \( L \). Fix an $s_0\in\mathbb{C}$. Let $S$ be a finite set of places of $K$ containing the archimedean places of $K$. Let $S_u$ be the set of places $\nu\notin S$ such that $\Pi_\nu$ is unramified. For each place \( \nu \in S_u \), choose the normalized spherical Whittaker function \( W_\nu \in \mathcal{W}(\Pi_\nu, \Psi_{n,\nu}) \) and set \( \Phi_{\nu} \in \mathcal{S}(K^n_\nu) \) to be the characteristic function of \( \mathcal{O}_{K_{\nu}}^n \).

Let $S_r$ be the set of places   $\nu\notin S$ such that $\Pi_\nu$ is ramified. Then for every place $\nu\in S_r$, we choose $W_\nu\in\mathcal{W}(\Pi_\nu,{\psi}_{n,\nu})$ and $\Phi_\nu\in\mathcal{S}(F^n_{\nu})$ such that 
$$\frac{Z(s_0,W_{\nu}, \Phi_{\nu})}{L(s_0,\Pi_\nu,\mathrm{As})}\neq 0.$$
This choice is justified by the fact that $L(s, \Pi_\nu, \mathrm{As})$ serves as the greatest common divisor of the family of local Flicker integrals when $\nu$ is inert in $L$ (\cite{MAT2009}), and as the greatest common divisor of the family of local Rankin-Selberg integrals when $\nu$ splits in $L$ (\cite{JPSS1983}).

Finally for $\nu\in S$, using Theorem \ref{main1}, we choose $\Phi_\nu$ and $W_{\nu}$ so that the local integrals $Z(s,W_\nu,\Phi_\nu)$ do not vanish at any complex number $s\in\mathbb{C}$. Choose $\varphi\in V_{\Pi}$ such that the associated global Whittaker function $W_{\varphi}$ is $\prod_{\nu} W_{\nu},$ and choose $\Phi\in\mathcal{S}(\mathbb{A}^n_K)$ to be $\Phi=\prod_{\nu}\Phi_\nu.$ Note that such a choice of $\varphi$ is possible due to the Fourier expansion of an automorphic form for $\operatorname{GL}_n$ \cite{PS1979,Sha1974}. Using Proposition \ref{3.5} and the Euler product decomposition, we obtain
\begin{equation}
    I(s,\Phi,\varphi)=L^{S}(s,\Pi,\mathrm{As})\prod\limits_{\nu\in S_r}\frac{Z(s,W_\nu,\Phi_\nu)}{L(s,\Pi_\nu,\mathrm{As})}\prod\limits_{\nu\in S}Z(s,W_\nu,\Phi_\nu).
\end{equation}
 Note that the finite product $\prod\limits_{\nu\in S}Z(s,W_\nu,\Phi_\nu)$ as a function of $s$ is nowhere vanishing by Theorem \ref{main1} and the finite product
 $$\prod\limits_{\nu\in S_r}\frac{Z(s,W_\nu,\Phi_\nu)}{L(s,\Pi_\nu,\mathrm{As})}$$
 is non-vanishing at $s_0\in\mathbb{C}$. We rewrite equation $3.5$ as
\[\frac{1}{\prod\limits_{\nu\in S_r}\frac{Z(s,W_\nu,\Phi_\nu)}{L(s,\Pi_\nu,\mathrm{As})}\prod\limits_{\nu\in S}Z(s,W_\nu,\Phi_\nu)}=\frac{L^{S}(s,\Pi,\mathrm{As})}{I(s,\Phi,\varphi)},\]
so that by our choice of ${\varphi}$ and $\Phi$, the quotient $\frac{L^{S}(s,\Pi,\mathrm{As})}{I(s,\Phi,\varphi)}$ is holomorphic at $s_0\in\mathbb{C}$. As the choice of $s_0$ was arbitrary, this quotient is, in fact, entire. By Proposition \ref{1.2}, if \( \omega_{\Pi} \) is nontrivial on \( \mathbb{I}^{1}_K \), then \( I(s,\Phi,\varphi) \) is entire for all choices of \( \Phi \) and \( \varphi \), implying that \( L^{S}(s,\Pi,\mathrm{As}) \) is also entire. Otherwise, \( I(s,\Phi,\varphi) \) can have at most simple poles at \( s = -i\delta \) and \( s = 1 - i\delta \), and consequently, so can \( L^{S}(s,\Pi,\mathrm{As}) \).

\section{Bump-Friedberg Integrals}\label{s4}

\subsection{The Bump-Friedberg integrals}\label{1.3.2}

Let \((\Pi, V_{\Pi})\) be a unitary cuspidal automorphic representation of \( G_n(\mathbb{A}_K) \) with central character \( \omega_{\Pi} \). Then \( \Pi \) is isomorphic to a restricted tensor product, \( \Pi \cong \bigotimes'_\nu \Pi_\nu \), taken over the places \( \nu \) of \( K \). Here, each \( \Pi_\nu \) belongs to \( \operatorname{Irr}_{gen}(G_n(K_\nu)) \) and is unramified for all but finitely many places \( \nu \). If $\omega_{\Pi}$ is trivial on $\mathbb{I}^{1}_K$, let $\delta$ be the real number such that $\omega_{\Pi}(.) = \left|.\right|^{in\delta}$. We introduce the Bump-Friedberg integrals as defined in \cite{BF1990}.

\begin{flushleft}
    {\textbf{Even Case $(n=2m)$}}
\end{flushleft}

For $\Phi\in\mathcal{S}(\mathbb{A}^m_K)$ and $\varphi\in V_{\Pi}$, consider the global integral
\[I(s_1,s_2,\Phi,\varphi)=\int\limits_{[G_m\times \>G_m]}E(g',s_2;\Phi,\omega_{\Pi})\>\varphi(J(g,g'))\>\left|{\frac{\operatorname{det}(g)}{\operatorname{det}(g')}}\right|_{\mathbb{A}_K}^{s_1-1/2}dg\>dg'.
 \]
We outline key properties of these integrals, drawing from \cite{BF1990}.
 \begin{prop}\label{1.4}
    The integral $I(s_1,s_2,\Phi,\varphi)$ is convergent whenever the Eisenstein series is holomorphic at $s_2$. It represents a meromorphic function in the variables $s_1$ and $s_2$.
\end{prop}

Theorem $2$ of \cite{BF1990} determines the possible poles of $I(s_1,s_2,\Phi,\varphi)$. 

\begin{prop}\label{1.5}
    The integral \( I(s_1,s_2,\Phi,\varphi) \) is entire if \( \omega_{\Pi} \) is nontrivial on \( \mathbb{I}^{1}_K \). Otherwise, it can have at most simple poles along the lines \( s_2 = -i\delta \) and \( s_2 = 1 - i\delta \).
\end{prop}
We recall the factorization of the global integral $I(s_1,s_2,\Phi,\varphi)$ (given in Section $1$ of \cite{BF1990}).
\begin{prop}\label{1.6}
    If $\varphi\in V_{\Pi}$ is a cusp form, let
    $$ W_{\varphi}(g)=\int\limits_{N_n(K)\backslash N_n(\mathbb{A}_{K})}\varphi(ng)\overline{\Psi'}(n)\>dn $$
    be the associated Whittaker function. For $\Phi\in\mathcal{S}(\mathbb{A}^m_{K})$, the integral $B(s_1,s_2,W_\varphi,\Phi)$ defined as 
    $$ 
    \int\limits_{N_m(\mathbb{A}_K)\backslash G_m(\mathbb{A}_K)}\int\limits_{N_m(\mathbb{A}_K)\backslash G_m(\mathbb{A}_K)}W_{\varphi}(J(g,g'))\Phi(e_m g')\left|\operatorname{det}(g)\right|_{\mathbb{A}_K}^{s_1-1/2}\left|\operatorname{det}(g')\right|_{\mathbb{A}_K}^{s_2-s_1+1/2}\>dg\>dg'
    $$
    converges when $\operatorname{Re}(s_1)$ and $\operatorname{Re}(s_2)$ are sufficiently large and, when this is the case, we have
    \[I(s_1,s_2,\Phi,\varphi)=B(s_1,s_2,W_\varphi,\Phi).\]
\end{prop}
Again, the global integrals are related to the local integrals for decomposable vectors. Let \( W_{\varphi} = \prod_{\nu} W_\nu \), where the product runs over all places \( \nu \) of \( K \), and each \( W_\nu \) lies in \( \mathcal{W}(\Pi_\nu, \Psi'_{n,\nu}) \). For almost all unramified places $\nu$, assume that \( W_\nu \) is the normalized spherical Whittaker function. Similarly, let \( \Phi = \prod_{\nu} \Phi_\nu \), where each \( \Phi_\nu \) is a Schwartz function in \( \mathcal{S}(K_\nu^m) \), and for almost all unramified places \( \nu \), \( \Phi_\nu \) is the characteristic function of \( \mathcal{O}_{K_{\nu}}^m \). When $\operatorname{Re}(s_1)$ and $\operatorname{Re}(s_2)$ are sufficiently large,
\[B(s_1,s_2,W_\varphi,\Phi)=\prod\limits_{\nu}B(s_1,s_2,W_\nu,\Phi_\nu),\]
where the local integral $B(s_1,s_2,W_\nu,\Phi_\nu)$ is equal to
\[\int\limits_{N_m(K_{\nu})\backslash G_m(K_{\nu})}\int\limits_{N_m(K_{\nu})\backslash G_m(K_{\nu})}W_{\nu}(J(g,g'))\Phi_\nu(e_m g')\left|\operatorname{det}(g)\right|_{K_\nu}^{s_1-1/2}\left|\operatorname{det}(g')\right|_{K_\nu}^{s_2-s_1+1/2}\>dg\>dg'.\]

The integrals $B(s_1,s_2,W_\nu,\Phi_\nu)$ converge absolutely when $\operatorname{Re}(s_1)$ and $\operatorname{Re}(s_2)$ are sufficiently large \cite{BF1990}.

\begin{flushleft}
    {\textbf{Odd Case $(n=2m+1)$}}
\end{flushleft}

For $\Phi\in\mathcal{S}(\mathbb{A}^m_K)$ and $\varphi\in V_{\Pi}$, consider the global integral 
\[I(s_1,s_2,\Phi,\varphi)=\int\limits_{[G_{m+1}\times \>G_m]}E\left(g,\frac{s_1+ms_2}{m+1};\Phi,\omega_{\Pi}\right)\>\varphi(J(g,g'))\>\left({\frac{\left|\operatorname{det}(g')\right|_{\mathbb{A}_K}}{\left|\operatorname{det}(g)\right|_{\mathbb{A}_K}^{m/m+1}}}\right)^{-s_1+s_2}dg\>dg'.
 \]
 We record Theorem $2$ of \cite{BF1990} below.
\begin{prop}\label{1.7}
    The integral $I(s_1,s_2,\Phi,\varphi)$ is everywhere convergent and admits an analytic continuation in the variables $s_1$ and $s_2$.
\end{prop}
We recall the factorization of the global integral $I(s_1,s_2,\Phi,\varphi)$ (given in Section $1$ of \cite{BF1990}).
\begin{prop}\label{1.8}
    If $\varphi\in V_{\Pi}$ is a cusp form, then let
    $$ W_{\varphi}(g)=\int\limits_{N_n(K)\backslash N_n(\mathbb{A}_{K})}\varphi(ng)\overline{\Psi'}(n)\>dn $$
    be the associated Whittaker function. For $\Phi\in\mathcal{S}(\mathbb{A}^m_{F})$, the integral $B(s_1,s_2,W_\varphi,\Phi)$ defined as 
    $$ 
    \int\limits_{N_m(\mathbb{A}_K)\backslash G_{m}(\mathbb{A}_K)}\int\limits_{N_{m+1}(\mathbb{A}_K)\backslash G_{m+1}(\mathbb{A}_K)}W_{\varphi}(J(g,g'))\Phi(e_{m+1} g)\left|\operatorname{det}(g)\right|_{\mathbb{A}_K}^{s_1}\left|\operatorname{det}(g')\right|_{\mathbb{A}_K}^{s_2-s_1}\>dg\>dg'
    $$
    converges when $\operatorname{Re}(s_1)$ and $\operatorname{Re}(s_2)$ are sufficiently large and, when this is the case, we have
    \[I(s_1,s_2,\Phi,\varphi)=B(s_1,s_2,W_\varphi,\Phi).\]
\end{prop}
As before, the global integrals factorize for decomposable vectors. Let \( W_{\varphi} = \prod_{\nu} W_\nu \), where the product runs over all places \( \nu \) of \( K \), and each \( W_\nu \) belongs to \( \mathcal{W}(\Pi_\nu, \Psi'_{n,\nu}) \). For almost all unramified places $\nu$, assume that \( W_\nu \) is the normalized spherical Whittaker function. Similarly, let \( \Phi = \prod_{\nu} \Phi_\nu \), where each \( \Phi_\nu \) is a Schwartz function in \( \mathcal{S}(K_\nu^m) \), and for almost all unramified places \( \nu \), \( \Phi_\nu \) is the characteristic function of \( \mathcal{O}_{K_{\nu}}^m \). When $\operatorname{Re}(s_1)$ and $\operatorname{Re}(s_2)$ are sufficiently large,
\[B(s_1,s_2,W_\varphi,\Phi)=\prod\limits_{\nu}B(s_1,s_2,W_\nu,\Phi_\nu),\]
where the local integral $B(s_1,s_2,W_\nu,\Phi_\nu)$ is equal to
\[\int\limits_{N_m(K_{\nu})\backslash G_m(K_{\nu})}\int\limits_{N_{m+1}(K_{\nu})\backslash G_{m+1}(K_{\nu})}W_{\varphi}(J(g,g'))\Phi(e_{m+1} g)\left|\operatorname{det}(g)\right|_{K_\nu}^{s_1}\left|\operatorname{det}(g')\right|_{K_\nu}^{s_2-s_1}\>dg\>dg'.\]

The integrals $B(s_1,s_2,W_\nu,\Phi_\nu)$ converge absolutely when $\operatorname{Re}(s_1)$ and $\operatorname{Re}(s_2)$ are sufficiently large \cite{BF1990}.

Now let \( n \) be arbitrary (even or odd, with no restriction). Below, we summarize the unramified computation for the local Bump-Friedberg integrals from \cite{BF1990}.
\begin{prop}\label{3.7}
    Let \(\nu\) be a place of \(K\), and suppose that \(\Pi_\nu\) is an unramified, non-archimedean representation. Consider the normalized spherical Whittaker function \( W_\nu \in \mathcal{W}(\Pi_\nu, {\Psi}'_{n,\nu}) \) and let \(\Phi_\nu \in \mathcal{S}(K_\nu^m)\) be the characteristic function of \( \mathcal{O}_{K_{\nu}}^m \). The Bump–Friedberg integral is then given by  

    \[
    B(s_1,s_2,W_\nu,\Phi_\nu) = L(s_1, \Pi_\nu)L(s_2, \Pi_\nu, \wedge^2).
    \]
\end{prop}

We now prove Theorem $\ref{main2}$ in the even case. The proof in the odd case is similar and is omitted.

\section*{Proof of Theorem \ref{main2}}
First assume that $E$ and $F$ are non-archimedean. In \cite{MY2013}, Miyauchi and Yamauchi showed that there exist $W\in\mathcal{W}(\pi,{\psi}'_{n})$ and $\Phi\in\mathcal{S}(F^n)$ such that 
    \[B(s_1,s_2,W,\Phi)=L(s_1,\pi)L_f(s_2,\pi,\wedge^2)\]
    where if the standard $L$-function of $\pi$ is written as
    \[
L(s_1, \pi) =\prod_{i=1}^n(1-\alpha_i q_F^{-s_1})^{-1},\ \alpha_i \in \C,
\]
then
$L_f(s_2,\pi,\wedge^2)$ denotes the formal exterior square $L$-function of $\pi_{\nu}$ defined by
\[
L_f(s_2,\pi,\wedge^2)
= \prod_{1 \leq i < j \leq n} (1-\alpha_i \alpha_j q_F^{-s_2})^{-1}.
\]

It is clear from the expressions of $L(s_1,\pi)$ and $L_f(s_2,\pi,\wedge^2)$ that $B(s_1,s_2,W,\Phi)$ is the reciprocal of a polynomial in $q_F^{-s_1}$ and $q_F^{-s_2}$. Hence, it is non-vanishing in $\mathbb{C}\times\mathbb{C}$.

Next, we assume that $E$ and $F$ are archimedean. Using Theorem $7.9$ and Proposition $7.10$ of \cite{HJ2024}, there exists a non-zero complex polynomial $p_{\pi,\wedge^2}$ (dependent on $\pi$), $W\in\mathcal{W}(\pi,{\psi}'_{n})$ and $\Phi\in\mathcal{S}(F^m)$ such that
\[{B(s_1,s_2,W,\Phi)}=p_{\pi,\wedge^2}(s_2)L(s_1,\pi)L(s_2,\pi,\wedge^2).\]
We claim that the product $p_{\pi,\wedge^2}(s_2)L(s_2,\pi,\wedge^2)$ does not vanish at any $s_2\in\C$. We assume that $\pi$ is ramified otherwise $p_{\pi,\wedge^2}\equiv 1$, and our claim would hold. 
For $\pi =\pi_1\times\pi_2\cdots\times\pi_r$ with $\pi_i$ essentially square-integrable, 
\[L(s_2,\pi,\wedge^2) = \prod_{j=1}^{r} L(s_2,\pi_j,\wedge^2) \prod_{1 \leq j < \ell \leq r} L(s_2, \pi_j \times \pi_{\ell})\]
by the multiplicativity of the $L$-factors \cite{MAT2017}. Therefore, by the proof of Proposition 7.10 in \cite{HJ2024},
\[p_{\pi,\wedge^2}(s_2)=\prod_{j = 1}^{r} p_{\pi_j,\wedge^2}(s_2) \prod_{1 \leq j < \ell \leq r} p_{\pi_j\times\pi_{\ell},\mathrm{RS}}(s_2),\]
where $p_{\pi_j,\wedge^2}$ is the polynomial as defined in Proposition $7.10$ of \cite{HJ2024}. 
So we can restrict ourselves to the case where $\pi$ is essentially square-integrable.

If $F=\mathbb{C}$ so that $\pi= e^{i\kappa \arg} |\cdot|_{\C}^{t}$, then $p_{\pi,\wedge^2}\equiv 1$. Similarly, if $F = \R$ and $\pi=\operatorname{sgn}^{\kappa} |\cdot|_{\R}^{t}$, then $p_{\pi,\wedge^2}\equiv 1$. Finally, if $F= \R$ and $\pi = D_{\kappa} \otimes \left|\det\right|_{\R}^{t}$, we have that
\[p_{\pi,\wedge^2}(s_2)L(s_2,\pi,\wedge^2)=\zeta_{\R}(s_2+2t+\kappa),\]
which is non-vanishing for any $s_2\in\mathbb{C}$. This completes the proof.

\begin{remark}
    The local Whittaker functions chosen in Theorems \ref{main1} and \ref{main2} correspond to local newforms. These include archimedean newforms as defined in \cite{HU2024} and non-archimedean newforms as outlined in \cite{JPSS1981}. Consequently, these newforms not only serve as weak test vectors (\cite{Jo2023},\cite{HJ2024}) for local Rankin-Selberg, Flicker, and Bump-Friedberg integrals but also ensure that their corresponding local integrals are non-vanishing everywhere. 
\end{remark}

\section*{Proof of Corollary \ref{coro2}}

Assume \(n\) is even. The argument for \(n\) odd is identical. Fix an $s_0\in\mathbb{C}$. Let $T$ be a finite set of places of $K$ containing the archimedean places of $K$. Let $T_u$ be the set of places $\nu\notin T$ such that $\Pi_\nu$ is unramified. For each place \( \nu \in T_u \), choose \( W_\nu \in \mathcal{W}(\Pi_\nu,{\Psi}'_{n,\nu}) \) as the normalized spherical Whittaker function, and let \( \Phi_{\nu} \in \mathcal{S}(K^m_\nu) \) be the characteristic function of \( \mathcal{O}_{K_{\nu}}^m \).

Let $T_r$ be the set of places $\nu\notin T$ such that $\Pi_\nu$ is ramified. For every place $\nu\in T_r$, we choose $W_\nu\in\mathcal{W}(\Pi_\nu,{\Psi}_{n,\nu})$ and $\Phi_\nu\in\mathcal{S}(K^m_{\nu})$ such that 
$$\frac{B(s_0,W_\nu,\Phi_\nu)}{L(s_0,\Pi_\nu,\mathrm{BF})}\neq 0.$$ Here $B(s,W_{\nu},\Phi_{\nu})$ denotes the integral $B(s,2s,W_\nu,\Phi_\nu)$. 
This choice is possible because $L(s,\Pi_\nu,\mathrm{BF})$ is the greatest common divisor of the family of local Bump-Friedberg integrals (\cite{MAT2013}). 
Finally for $\nu\in T$, using Theorem \ref{main2}, we choose $\Phi_\nu$ and $W_{\nu}$ so that the local integrals $B(s,W_\nu,\Phi_\nu)$ do not vanish at any complex number $s\in\mathbb{C}$. 

Choose $\varphi\in V_{\Pi}$ such that the associated global Whittaker function $W_{\varphi}$ is 
$\prod_{\nu}W_{\nu},$
and choose $\Phi\in\mathcal{S}(\mathbb{A}^m_F)$ to be $\Phi=\prod_{\nu}\Phi_\nu.$ Once again, such a choice of $\varphi$ is possible due to the Fourier expansion of an automorphic form for $\operatorname{GL}_n$ \cite{PS1979, Sha1974}. Using Proposition \ref{3.7} and the Euler product decomposition, we obtain
\begin{equation}
    I(s,\Phi,\varphi)\coloneq I(s,2s,\Phi,\varphi)=L^{T}(s,\Pi,\mathrm{BF})\prod\limits_{\nu\in T_r}\frac{B(s,W_\nu,\Phi_\nu)}{L(s,\Pi_\nu,\mathrm{BF})}\prod\limits_{\nu\in T}B(s,W_\nu,\Phi_\nu).
\end{equation}
 
 Note that the finite product $\prod\limits_{\nu\in T}B(s,W_\nu,\Phi_\nu)$ as a function of $s$ is nowhere vanishing by Theorem \ref{main2} and the finite product
 $$\prod\limits_{\nu\in T_r}\frac{B(s,W_\nu,\Phi_\nu)}{L(s,\Pi_\nu,\mathrm{BF})}$$
 is non-vanishing at $s_0\in\mathbb{C}$.

We can rewrite equation $4.8$ as
\[\frac{1}{\prod\limits_{\nu\in T_r}\frac{B(s,W_\nu,\Phi_\nu)}{L(s,\Pi_\nu,\mathrm{BF})}\prod\limits_{\nu\in T}B(s,W_\nu,\Phi_\nu)}=\frac{L^{T}(s,\Pi,\mathrm{BF})}{I(s,\Phi,\varphi)}\]
so that by our choice of ${\varphi}$ and $\Phi$, the quotient $\frac{L^{T}(s,\Pi,\mathrm{BF})}{I(s,\Phi,\varphi)}$ is holomorphic at $s_0\in\mathbb{C}$. As the choice of $s_0$ was arbitrary, this quotient is, in fact, entire. By Proposition \ref{1.5}, if \( \omega_{\Pi} \) is nontrivial on \( \mathbb{I}^{1}_K \), then \( I(s,\Phi,\varphi) \) is entire for all choices of \( \Phi \) and \( \varphi \), implying that \( L^{T}(s,\Pi,\mathrm{BF}) \) is also entire. Otherwise, \( I(s,\Phi,\varphi) \) can have at most simple poles at \( s = -\frac{i\delta}{2} \) and \( s = \frac{1 - i\delta}{2} \), and consequently, so can \( L^{T}(s,\Pi,\mathrm{BF}) \).

\section*{Acknowledgements}  I would like to express my sincere gratitude to my advisor, Professor Ravi Raghunathan, for his unwavering encouragement. I am also grateful to Professor Dipendra Prasad for his valuable suggestions, which significantly improved this paper.

\section*{Funding}
The author gratefully acknowledges the financial support provided by the Ministry of Human Resource Development (MHRD) during the course of his Ph.D. studies.



\end{document}